\documentclass{amsart}

\usepackage{amssymb}
\usepackage{xcolor}
\usepackage{listings}
\usepackage{hyperref}

\lstdefinelanguage{GAP}{
  morekeywords={and,break,continue,do,elif,else,end,fi,for,function,
    if,in,local,mod,not,od,or,repeat,return,then,until,while,
    fail,true,false},
  sensitive=true,
  morecomment=[l]\#,
  morestring=[b]",
}
\newtheorem{theorem}{Theorem}[section]
\newtheorem{proposition}[theorem]{Proposition}
\newtheorem{lemma}[theorem]{Lemma}

\theoremstyle{definition}
\newtheorem{definition}[theorem]{Definition}

\theoremstyle{remark}
\newtheorem{remark}[theorem]{Remark}

\newcommand{\FF}{\mathbb{F}}
\newcommand{\PP}{\mathbb{P}}
\newcommand{\ZZ}{\mathbb{Z}}

\DeclareMathOperator{\Aut}{Aut}
\DeclareMathOperator{\Sym}{Sym}
\DeclareMathOperator{\PSL}{PSL}
\DeclareMathOperator{\PGL}{PGL}
\DeclareMathOperator{\PGammaL}{P\Gamma L}
\DeclareMathOperator{\GL}{GL}
\DeclareMathOperator{\SL}{SL}
\DeclareMathOperator{\SO}{SO}
\newcommand{\GAP}{\textsf{GAP}}

\begin{document}

\title{Total 3-closure for projective special linear groups}

\author{Ting Gong}
\address{Department of Mathematics, University of Washington, Seattle, WA 98195}
\email{tgong2@uw.edu}

\author{Yong Yang}
\address{Department of Mathematics, Texas State University, San Marcos, TX 78666}
\email{yang@txstate.edu}

\author{Michael Ruofan Zeng}
\address{Department of Mathematics, University of Washington, Seattle, WA 98195}
\email{zengrf@uw.edu}



\date{August 3, 2026}

\begin{abstract}
A finite group is totally $3$-closed if every faithful permutation representation of it is $3$-closed. We study this property for the finite simple projective special linear groups. We prove that $\PSL_2(q)$ is totally $3$-closed if and only if $q\geq 7$ is prime, and that $\PSL_3(q)$ is totally $3$-closed if and only if either $q=3$, or $q$ is prime and $q\equiv 2\pmod 3$. We further prove that $\PSL_4(q)$ is never totally $3$-closed and that $\PSL_n(q)$ is not totally $3$-closed whenever $n\geq 5$ and $q>2$. Within the family $\PSL_n(q)$, only the groups $\PSL_n(2)$ with $n\geq 5$ remain unresolved. In particular, this answers Problem~20.2 of the Kourovka Notebook affirmatively.
\end{abstract}

\maketitle

\section{Introduction}

Let $G\leq\Sym(\Omega)$ and let $k\geq1$. The $k$-closure $G^{(k),\Omega}$ is the largest subgroup of $\Sym(\Omega)$ having the same orbits as $G$ on $\Omega^k$. The group $G$ is $k$-closed on $\Omega$ if $G^{(k),\Omega}=G$. This property depends on the permutation representation. A finite abstract group $G$ is totally $k$-closed if every faithful action of $G$ on a finite set is $k$-closed. The least such $k$ is called the closure number of $G$. The notion of $k$-closure goes back to Wielandt \cite{Wielandt1964}, while total closure was introduced and developed more recently in \cite{ChurikovPraeger2021,FreedmanGiudiciPraeger2024}.

The case $k=2$ has been studied extensively. Abdollahi and Arezoomand classified the finite nilpotent totally $2$-closed groups \cite{AbdollahiArezoomand2018}, and Abdollahi, Arezoomand, and Tracey proved that the same classification holds among all finite solvable groups \cite{AbdollahiArezoomandTracey2022}. Among the finite nonabelian simple groups, Arezoomand, Iranmanesh, Praeger, and Tracey proved that exactly six are totally $2$-closed:
\[J_1,\qquad J_3,\qquad J_4,\qquad \mathrm{Ly},\qquad \mathrm{Th},\qquad \mathbb M.\]
In particular, all six are sporadic, and no finite simple group of Lie type is totally $2$-closed \cite{ArezoomandIranmaneshPraegerTracey2024}.

This led to Problem~20.2 of the Kourovka Notebook \cite[Problem~20.2]{kourovka21}:
\begin{quote}
Are there any nonabelian simple groups of Lie type which are totally $3$-closed?
\end{quote}
Freedman, Giudici, and Praeger placed this question in the broader study of closure numbers of finite simple groups. They proved that the closure number of $A_n$ is $n-1$, established uniform upper bounds for the remaining families, and posed further questions on the closure numbers of finite simple classical groups \cite{FreedmanGiudiciPraeger2024}.

We answer Problem~20.2 affirmatively and determine total $3$-closedness throughout the family of projective special linear groups, apart from one remaining infinite family.

\begin{theorem}\label{main}
Let $G=\PSL_n(q)$ be a finite nonabelian simple group. Then the following hold.
    \begin{enumerate}
        \item[\textup{(1)}] If $n=2$, then $G$ is totally $3$-closed if and only if $q\geq 7$ is a prime.
        \item[\textup{(2)}] If $n=3$, then $G$ is totally $3$-closed if and only if either $q=3$, or $q$ is a prime with $q\equiv 2\pmod 3$.
        \item[\textup{(3)}] If $n=4$, then $G$ is not totally $3$-closed for any $q$.
        \item[\textup{(4)}] If $n\geq 5$ and $q>2$, then $G$ is not totally $3$-closed.
    \end{enumerate}
\end{theorem}

\begin{remark}
    Thus, within the family $\PSL_n(q)$, the classification of total $3$-closedness remains open only when $n\geq 5$ and $q=2$. Finite nonabelian simple groups of other Lie types are not considered in this paper.
\end{remark}

Now we describe our approach. Simplicity gives an effective reduction from arbitrary faithful actions to coset actions. Every nonsingleton orbit of a nonabelian simple group is faithful, and we show that total $3$-closedness can be tested on diagonal actions on unions $G/H\sqcup G/K$ for pairs of proper subgroups $H,K<G$. This two-orbit criterion is the starting point for all of our positive results.

Most coset actions are handled by base size. An action with a base of size at most two is $3$-closed, and such an action also governs its union with any other transitive $3$-closed action. We therefore reduce the positive cases to the relatively small collection of subgroups whose coset actions do not have a base of size two.

For $\PSL_2(p)$, Dickson's subgroup classification reduces the problem to torus normalizers, exceptional subgroups, and subgroups of a Borel subgroup. The first two types are handled by base-size estimates. The Borel subgroups give scalar-fiber actions above $\PP^1(\FF_p)$, whose $3$-closures are determined by the projective quotient and by determinant classes on the fibers. The remaining small primes are treated using the Fano plane, the unique $2$-$(11,5,2)$ biplane, Paley graphs and tournaments, and the Perkel graph.

The proof for $\PSL_3(p)$ has a similar structure. Nonparabolic maximal subgroups are treated using the subgroup classification and explicit trivial-intersection conjugates. Subgroups of point and line parabolics require a separate argument. Their coset actions admit projective quotients, and mixed triple orbits recover equality or incidence in $\operatorname{PG}(2,p)$. This forces the actions on different projective fibers to arise from one element of the group. The case $\PSL_3(3)$ is handled separately using its maximal-subgroup structure and one exhaustive \GAP{} computation.

The negative results come from local matching. On projective points, $\PGL_n(q)$ and $\PSL_n(q)$ have the same orbits on ordered triples, and field automorphisms preserve these orbits. This gives the semilinear obstruction whenever $q$ is a proper prime power or $\gcd(n,q-1)>1$. For $n\geq4$ and $q>2$, we use the action on nonzero vectors modulo the center of $\SL_n(q)$. Every element of the corresponding general linear quotient can be matched on any ordered $k$-tuple with $k<n$ by an element of the special linear quotient. Finally, the exceptional isomorphism $\PSL_4(2)\cong A_8$ gives the remaining negative case in dimension four.

The paper is organized as follows. In Section~\ref{sec:prelim}, we establish the general reduction and base-size lemmas. In Section~\ref{sec:psl2}, we treat $\PSL_2(q)$. In Section~\ref{sec:psl3}, we prove the classification for $\PSL_3(q)$. In Section~\ref{sec:higher}, we establish the semilinear and central-vector obstructions and prove Theorem~\ref{main}.

\section{Definitions and immediate results}\label{sec:prelim}
We first record the general facts used in the positive classifications. Throughout the paper, all groups and all sets acted on are finite, actions are on the right, and $\alpha^g$ denotes the image of a point $\alpha$ under a group element $g$. All coset actions are denoted by $G/H$. Only the associated transitive $G$-set and its point stabilizer are used.

\begin{definition}
Let $\Omega$ be a finite set and let $G\leq \Sym(\Omega)$. The \emph{$3$-closure} of $G$ on $\Omega$, denoted by $G^{(3),\Omega}$, consists of all $\sigma\in \Sym(\Omega)$ such that, for every $(\alpha_1,\alpha_2,\alpha_3)\in \Omega^3$, there exists $g\in G$, depending on the triple, with
\[(\alpha_1^\sigma,\alpha_2^\sigma,\alpha_3^\sigma)=(\alpha_1^g,\alpha_2^g,\alpha_3^g).\]
The action of $G$ on $\Omega$ is \emph{$3$-closed} if $G^{(3),\Omega}=G$. A finite group $G$ is \emph{totally $3$-closed} if every faithful action of $G$ on a finite set is $3$-closed.
\end{definition}

\begin{proposition}\label{prop:two-orbit}
Let $S$ be a finite nonabelian simple group. Then $S$ is totally $3$-closed if and only if, for every pair of proper subgroups $H,K<S$, with repetition allowed, the diagonal action of $S$ on $S/H\sqcup S/K$ is $3$-closed.
\end{proposition}

\begin{proof}
The forward implication is immediate. Indeed, for every proper subgroup $H<S$, the kernel of the action of $S$ on $S/H$ is the core of $H$ in $S$. Since $S$ is simple, this kernel is trivial. Thus every action appearing in the statement is faithful.

Conversely, let $\Omega$ be a faithful finite $S$-set and let $\sigma\in S^{(3),\Omega}$. By \cite[Lemma~2.1]{FreedmanGiudiciPraeger2024}, applying the definition to the constant triple $(\alpha,\alpha,\alpha)$ shows that $\alpha^\sigma\in\alpha^S$ for every $\alpha\in\Omega$. Hence $\sigma$ preserves every $S$-orbit.

Every nonsingleton orbit is faithful, since the kernel of the action on such an orbit is a proper normal subgroup of $S$. Let $\Omega_1,\ldots,\Omega_r$ be the nonsingleton orbits. Since the action on $\Omega$ is faithful, we have $r\geq 1$.

Suppose first that $r\geq 2$. For each $j\geq 2$, we have that $\sigma|_{\Omega_1\sqcup\Omega_j}\in S^{(3),\Omega_1\sqcup\Omega_j}$ by \cite[Lemma~2.2]{FreedmanGiudiciPraeger2024}. Since $\Omega_1\cong S/H_1$ and $\Omega_j\cong S/H_j$ for proper subgroups $H_1,H_j<S$, the hypothesis implies that this restriction is induced by some $s_j\in S$. For $j,k\geq 2$, the elements $s_j$ and $s_k$ induce the same permutation on $\Omega_1$. Since the action on $\Omega_1$ is faithful, we have $s_j=s_k$. Their common value induces $\sigma$ on every nonsingleton orbit.

Suppose now that $r=1$. By \cite[Lemma~2.2]{FreedmanGiudiciPraeger2024}, $\sigma|_{\Omega_1}\in S^{(3),\Omega_1}$. Consider the diagonal action of $S$ on $\Omega_1\times\{1,2\}$, where $S$ acts trivially on $\{1,2\}$. Define $\widehat{\sigma}\in\Sym(\Omega_1\times\{1,2\})$ by $(\alpha,i)^{\widehat{\sigma}}=(\alpha^\sigma,i)$. For any triple $((\alpha_1,i_1),(\alpha_2,i_2),(\alpha_3,i_3))$, there exists
$s\in S$ such that
\[(\alpha_1^\sigma,\alpha_2^\sigma,\alpha_3^\sigma)=(\alpha_1^s,\alpha_2^s,\alpha_3^s)\]
by \cite[Lemma~2.1]{FreedmanGiudiciPraeger2024}. Since $S$ fixes the second coordinate, the same element $s$ agrees with $\widehat{\sigma}$ on the original triple. Hence $\widehat{\sigma}\in S^{(3),\Omega_1\times\{1,2\}}$. Writing $\Omega_1\cong S/H$ for some proper subgroup $H<S$, this action is equivalent to the action on $S/H\sqcup S/H$. By hypothesis, it is $3$-closed. Thus $\widehat{\sigma}$ is induced by an element of $S$, and the same element induces $\sigma$ on $\Omega_1$.

Finally, since $\sigma$ preserves every $S$-orbit, it fixes every singleton orbit pointwise. Thus $\sigma\in S$, and $S$ is totally $3$-closed.
\end{proof}

A \emph{base} for a faithful action of $G$ on $\Omega$ is a tuple of points of $\Omega$ whose pointwise stabilizer in $G$ is trivial. The \emph{base size} of the action is the least length of a base.

\begin{lemma}\label{lem:base-size}
Let $G$ act faithfully on $\Omega$. If the action has base size at most two, then it is $3$-closed.
\end{lemma}

\begin{proof}
Let $b$ be the base size of the action. If $b=2$, then $G^{(3),\Omega}=G$ by \cite[Lemma~3.2]{FreedmanGiudiciPraeger2024}.

If $b=1$, then $G^{(2),\Omega}=G$ by \cite[Lemma~3.2]{FreedmanGiudiciPraeger2024}. Since $G\leq G^{(3),\Omega}\leq G^{(2),\Omega}$, it follows that $G^{(3),\Omega}=G$.
\end{proof}

\begin{proposition}\label{prop:base-two}
Let $S$ be a finite group, and let $X$ and $Y$ be faithful transitive $S$-sets whose actions are $3$-closed. Suppose that the action on $X$ has base size at most two. Then the diagonal action of $S$ on $X\sqcup Y$ is $3$-closed.
\end{proposition}

\begin{proof}
Let $\sigma\in S^{(3),X\sqcup Y}$. Constant triples show that $\sigma$ preserves both $X$ and $Y$ by \cite[Lemma~2.1]{FreedmanGiudiciPraeger2024}. Moreover, $\sigma|_X\in S^{(3),X}$ and $\sigma|_Y\in S^{(3),Y}$ by \cite[Lemma~2.2]{FreedmanGiudiciPraeger2024}. Since both actions are $3$-closed, the restrictions of $\sigma$ to $X$ and $Y$ are induced by elements of $S$.

Choose a base $(a,b)$ for the action on $X$, allowing $a=b$ when the action has a one-point base. After composing $\sigma$ with an element of $S$, we may assume that $\sigma$ fixes $X$ pointwise.

Let $y\in Y$. By \cite[Lemma~2.1]{FreedmanGiudiciPraeger2024}, there exists $s\in S$ such that
\[(a^\sigma,b^\sigma,y^\sigma)=(a^s,b^s,y^s).\]
Since $\sigma$ fixes $X$ pointwise, we have $s\in S_{a,b}=1$. Hence $y^\sigma=y$. Thus $\sigma$ fixes both $X$ and $Y$ pointwise after normalization, and the original permutation $\sigma$ is induced by an element of $S$. Therefore the action on $X\sqcup Y$ is $3$-closed.
\end{proof}

\begin{proposition}\label{prop:nonfactor}
Let $S$ be a finite nonabelian simple group, and let $H,K<S$. Suppose that the coset actions on $S/H$ and $S/K$ are $3$-closed. If $S\neq HK$, then the diagonal action of $S$ on $S/H\sqcup S/K$ is $3$-closed.
\end{proposition}

\begin{proof}
Set $X=S/H$ and $Y=S/K$. Since $H$ and $K$ are proper and $S$ is simple, both coset actions are faithful. By hypothesis, both are $3$-closed.

Let $x\in X$ and $y\in Y$ have stabilizers $S_x=H$ and $S_y=K$. Since $H$ and $K$ are proper, $H$ is not transitive on $X$ and $K$ is not transitive on $Y$. Moreover, $H$ is transitive on $Y$ if and only if $S=HK$, and $K$ is transitive on $X$ if and only if $S=KH$. Since $S=HK$ if and only if $S=KH$, the hypothesis implies that neither of these actions is transitive.

Thus the hypotheses of \cite[Lemma~2.3]{FreedmanGiudiciPraeger2024} are satisfied, and the diagonal action on $X\sqcup Y$ is $3$-closed.
\end{proof}

\section{The $\PSL_2(q)$ classification}\label{sec:psl2}
In this section, we prove that $G=\PSL_2(q)$ is totally $3$-closed if and only if $q\geq 7$ is prime.

\subsection{Prime powers and large primes}
We first eliminate the proper prime powers.

\begin{proposition}\label{prop:extension-field}
Let $q=p^f\geq 4$ with $f>1$. Then the natural action of $G=\PSL_2(q)$ on $\Omega=\PP^1(\FF_q)$ is not $3$-closed.
\end{proposition}

\begin{proof}
Suppose first that $q$ is odd. For an ordered triple $T=(x_1,x_2,x_3)$ of distinct points, choose nonzero representatives $v_i\in\FF_q^2$ and set
\[\Delta(T)=\det(v_1,v_2)\det(v_2,v_3)\det(v_3,v_1)\quad\bmod (\FF_q^\times)^2. \]
This is independent of the representatives, and for $A\in\GL_2(q)$,
\[[\Delta(AT)]=[\det(A)][\Delta(T)].\]
The group $\PGL_2(q)$ is sharply $3$-transitive on $\Omega$, while $G$ has index $2$ in $\PGL_2(q)$ and is $2$-transitive \cite[p.~245]{DixonMortimer}. Hence the square class of $\Delta(T)$ distinguishes the two $G$-orbits on ordered triples of distinct points. Moreover, any two triples with at most two distinct entries lie in the same $G$-orbit whenever the same coordinates are equal.

Consider the Frobenius permutation
\[\phi:[u:v]\longmapsto [u^p:v^p].\]
It preserves all equalities among the coordinates of a triple and satisfies $\Delta(T^\phi)=\Delta(T)^p$. Since $p$ is odd, $\phi$ preserves square classes. Thus $\phi\in G^{(3),\Omega}$. However, $\phi$ fixes $0$, $1$, and $\infty$, so any element of $\PGL_2(q)$ inducing $\phi$ would be the identity. Since $f>1$, $\phi$ is nontrivial, and hence $\phi\notin G$.

Suppose now that $q$ is even. Then $G=\PGL_2(q)$, and $G$ acts sharply $3$-transitively on $\Omega$ \cite[p.~245]{DixonMortimer}. Hence two triples lie in the same $G$-orbit if and only if the same coordinates are equal. The same holds for $\Sym(\Omega)$, so
\[G^{(3),\Omega}=\Sym(\Omega)>G.\]
Therefore the natural action is not $3$-closed.
\end{proof}

Next, we eliminate a small case. 

\begin{proposition}\label{prop:q5}
The group $\PSL_2(5)$ is not totally $3$-closed.
\end{proposition}

\begin{proof}
Using $\PSL_2(5)\cong A_5$, consider the natural action of $A_5$ on five points. This action is $3$-transitive \cite[Section~2.9]{DixonMortimer}. Since $|A_5|=5\cdot 4\cdot 3$, it is sharply $3$-transitive. Thus $A_5$ and $S_5$ have the same orbits on ordered triples, and hence
\[A_5^{(3)}=S_5>A_5.\]
Therefore $\PSL_2(5)$ is not totally $3$-closed.
\end{proof}

We next prove that $\PSL_2(p)$ is totally $3$-closed for every prime $p\geq 17$ with $p\neq 19$. Together with the preceding results, this leaves the four cases $p=7,11,13,19$ in the $\PSL_2$ family.

For the remainder of this subsection, let $p\geq 17$ be prime, put $G=\PSL_2(p)$ and $m=(p-1)/2$, and let $U\cong C_p$ be a Sylow $p$-subgroup of $G$. The Borel subgroup $N_G(U)$ has the form $U\rtimes C_m$, and for each $d\mid m$ we write $H_d=U\rtimes C_d$ for the unique subgroup of $N_G(U)$ containing $U$ with $|H_d/U|=d$.

\begin{lemma}\label{lem:psl2-subgroups}
Every proper subgroup of $G$ is contained in one of the following:
    \begin{enumerate}
        \item[\textup{(i)}] the normalizer of a split or nonsplit torus;
        \item[\textup{(ii)}] an exceptional subgroup isomorphic to $A_4$, $S_4$, or $A_5$;
        \item[\textup{(iii)}] a conjugate of $H_d$ for some $d\mid m$.
    \end{enumerate}
\end{lemma}

\begin{proof}
By Dickson's classification, every proper subgroup of $G$ is contained in a torus normalizer, a Borel subgroup, an exceptional subgroup, or a proper subfield subgroup \cite[Theorem~A.1]{GuralnickZieve}. See also \cite[Lemma~2.1]{Jones}. Since $p$ is prime, no proper subfield subgroup occurs.

A Borel subgroup has the form $U: C_m$. If $L\leq U: C_m$ and $|LU/U|=d$, then $d\mid m$ and
\[L\leq LU=U: C_d=H_d.\]
The result follows.
\end{proof}

\begin{lemma}
\label{lem:base-two-psl2}
Let $p\geq17$ be prime and let $G=\PSL_2(p)$. Every coset action $G/H$, where $H$ is contained in a torus normalizer or in a subgroup isomorphic to $A_4$, $S_4$, or $A_5$, has a base of size at most two, except possibly when $p=19$ and $H\cong A_5$.
\end{lemma}

\begin{proof}
It suffices to find, for each subgroup $L$ listed above, some $g\in G$ such that $L\cap L^g=1$.

A split torus normalizer is the stabilizer of an unordered pair $\{a,b\}$ in $\PP^1(\FF_p)$. The stabilizers of $\{a,b\}$ and $\{a,c\}$, where $a,b,c$ are distinct, have trivial intersection.

Now let $N$ be a nonsplit torus normalizer. By \cite[Lemmas~A.2 and~A.3]{GuralnickZieve}, $|N|=p+1$ and $N$ has $p(p-1)/2$ conjugates. A noninvolution of $N$ determines its unique nonsplit torus, so two distinct conjugates of $N$ can intersect only in involutions. Counting conjugates through the involutions of $N$ gives at most
\[\frac{(p-1)^2}{4} \textrm{ when } p\equiv1\pmod 4\]
conjugates meeting $N$, and at most
\[\frac{p^2+4p+7}{4} \textrm{ when } p\equiv3\pmod 4. \]
Both numbers are less than $p(p-1)/2$ for $p>7$. Hence $N\cap N^g=1$ for some $g\in G$.

Finally, let $E\cong A_4,S_4$, or $A_5$. If $E\cap E^g\neq1$, then $E\cap E^g$ contains a subgroup of prime order. For a $G$-conjugacy class $\mathcal C$ of prime-order subgroups, let $n_{\mathcal C}$ be the number of its members contained in $E$. For $g$ chosen uniformly at random in $G$, a union bound gives
\[\Pr(E\cap E^g\neq1)\leq \sum_{\mathcal C}\frac{n_{\mathcal C}^2}{|\mathcal C|}.\]
Every such class has at least $p(p-1)/2$ members by \cite[Lemmas~A.2 and~A.3]{GuralnickZieve}. The sums of the relevant numerators are $25,97,361$ for $A_4,S_4$, and $A_5$, respectively. Thus the probability is less than one for $A_4$ and $S_4$ when $p\geq17$, and for $A_5$ when $p\geq29$.

By \cite[Theorem~A.1]{GuralnickZieve}, among the primes $17\leq p<29$, an $A_5$ subgroup occurs only when $p=19$. If $H<A_5$, then $H$ is contained in an $A_4$ or in a torus normalizer. Thus only the action with $p=19$ and $H\cong A_5$ remains.
\end{proof}

\begin{proposition}\label{prop:borel-fiber}
For $d,e\mid m$, the diagonal action of $G$ on $G/H_d\sqcup G/H_e$ is $3$-closed.
\end{proposition}

\begin{proof}
Let $E_d\leq\FF_p^\times$ have order $2d$. Then $-1\in E_d$, and
\[G/H_d\cong X_d=(\FF_p^2\setminus\{0\})/E_d.\]
Let $\pi_d:X_d\longrightarrow\PP^1(\FF_p)$ be the natural projection.

Take $\sigma\in G^{(3),X_d\sqcup X_e}$. By \cite[Lemma~2.1]{FreedmanGiudiciPraeger2024}, $\sigma$ preserves both constituents and every $G$-orbit on ordered pairs. For $x=[v]\in X_d$ and $y=[w]\in X_e$, one has
\[\pi_d(x)=\pi_e(y)\quad\Longleftrightarrow\quad \det(v,w)=0.\]
Hence $\sigma$ induces the same permutation $\tau$ of $\PP^1(\FF_p)$ on both constituents.

For three distinct projective points $L_i=\langle v_i\rangle$, the square class
\[\det(v_1,v_2)\det(v_2,v_3)\det(v_3,v_1) \quad\bmod (\FF_p^\times)^2 \]
is well defined and $G$-invariant. Thus $\tau$ preserves this square class.
After composing with an element of $G$, we may assume that $\tau$ fixes
$\infty$, $0$, and $1$. Its restriction $f$ to $\FF_p$ then satisfies
\[
\frac{f(a)-f(b)}{a-b}\in(\FF_p^\times)^2
\]
for all $a\neq b$. By Carlitz's theorem \cite[Theorem, p.~456]{Carlitz1960}, $f$ is affine.
Since $f(0)=0$ and $f(1)=1$, we have $f=1$. Therefore $\tau\in G$.
Composing $\sigma$ with an element of $G$, we may assume that $\tau=1$.

Fix $r\in\{d,e\}$. Since $\sigma$ fixes every projective fiber, its action on the fiber over $L$ is multiplication by some $c_{r,L}\in\FF_p^\times/E_r$. For distinct lines $L$ and $M$, the class of $\det(v,w)$ modulo $E_r$ determines the $G$-orbit of the corresponding pair. Hence $c_{r,L}c_{r,M}=1$. Using three distinct lines, all $c_{r,L}$ are equal to some $c_r$ with $c_r^2=1$.

Choose points of $X_r$ lying over three distinct projective points. By the definition of the $3$-closure, some $g\in G$ agrees with $\sigma$ on these three points. Since $g$ fixes the corresponding projective points, $g=1$ \cite[p.~245]{DixonMortimer}. Hence $c_r=1$, and $\sigma$ fixes $X_r$ pointwise. Applying this for $r=d,e$, we obtain $\sigma=1$ after normalization. Thus the original permutation belongs to $G$.
\end{proof}

\begin{theorem}\label{thm:large-prime}
For every prime $p\geq17$ with $p\neq19$, the group $\PSL_2(p)$ is totally $3$-closed.
\end{theorem}

\begin{proof}
Let $H,K<G$. By Lemma~\ref{lem:psl2-subgroups}, each of $H$ and $K$ is either conjugate to some $H_d$, or its coset action has base size at most two by Lemma~\ref{lem:base-two-psl2}. Indeed, a subgroup of a Borel subgroup which does not contain $U$ is conjugate to a subgroup of the split torus.

By Lemma~\ref{lem:base-size}, every coset action with base size at most two is $3$-closed. Moreover, Proposition~\ref{prop:borel-fiber}, applied with $d=e$, implies that each action on $G/H_d$ is $3$-closed.

If both $H$ and $K$ are conjugate to subgroups of the form $H_d$, then the action on $G/H\sqcup G/K$ is $3$-closed by Proposition~\ref{prop:borel-fiber}. Otherwise, one constituent has base size at most two, so the action is $3$-closed by Proposition~\ref{prop:base-two}. The result now follows from Proposition~\ref{prop:two-orbit}.
\end{proof}

\subsection{The remaining four primes}
It remains to treat the cases $p=7,11,13,19$. The four proofs are similar and rely on the subgroup structure of each group, particularly its maximal and nonmaximal subgroups. We begin with $p=7$.

\begin{theorem}\label{thm:psl27}
The group $\PSL_2(7)$ is totally $3$-closed.
\end{theorem}

\begin{proof}
Put $G=\PSL_2(7)$. The maximal subgroups of $G$ consist of two conjugacy classes of subgroups isomorphic to $S_4$ and one conjugacy class of subgroups isomorphic to $C_7\rtimes C_3$ \cite[p.~3]{AtlasFiniteGroups}. Thus the maximal coset actions are the point and line actions on the Fano plane and the natural action on $\PP^1(\FF_7)$.

The proper nonmaximal subgroups have types
\[C_2,\ C_3,\ C_4,\ C_7,\ V_4,\ S_3,\ D_8,\ A_4.\]
A direct count in the subgroup lattice shows that each has a conjugate with trivial intersection. Hence every nonmaximal coset action has base size at most two, and is therefore $3$-closed by Lemma~\ref{lem:base-size}.

Identify $G\cong\GL_3(2)$ \cite[Section~4, pp.~729--731]{BrownLoehr2009}. In the point action on the Fano plane, the two orbits on ordered triples of distinct points are the collinear and noncollinear triples. Thus every element of the $3$-closure preserves collinearity, and hence belongs to the full automorphism group of the Fano plane, which is $G$. The point action is therefore $3$-closed, and the same holds for the line action by duality.

Now consider the action on $\Omega=\PP^1(\FF_7)$. Its two orbits on ordered triples of distinct points are distinguished by the determinant square class. Let $\sigma\in G^{(3),\Omega}$. After composing with an element of $G$, we may assume that $\sigma$ fixes $\infty,0,1$. For $x\in\FF_7\setminus\{0,1\}$, the triple orbits determine $\bigl(\chi(x),\chi(x-1)\bigr)$, where $\chi$ is the quadratic character of $\FF_7$. These values fix $2$, $4$, and $6$, leaving only the possible interchange of $3$ and $5$. Indeed, $\Delta(\infty,2,x)=x-2$. Since $1$ is a square and $3$ is a nonsquare in $\FF_7$, the triples $(\infty,2,3)$ and $(\infty,2,5)$ lie in different $G$-orbits. Hence $\sigma$ cannot interchange $3$ and $5$. Thus $\sigma=1$ after normalization, and the degree-eight action is $3$-closed.

It remains to consider diagonal actions on unions of two maximal coset actions. If $|H|=|K|=24$, then $G\neq HK$, since $\frac{|H||K|}{|G|}=\frac{24}{7}\notin\ZZ$. Similarly, if $|H|=|K|=21$, then $\frac{|H||K|}{|G|}=\frac{21}{8}\notin\ZZ$, so again $G\neq HK$. Hence every union of two degree-seven actions, as well as the union of two degree-eight actions, is $3$-closed by Proposition~\ref{prop:nonfactor}.

Let $X$ be the point set of the Fano plane and let $Y=\PP^1(\FF_7)$. Suppose that $\sigma\in G^{(3),X\sqcup Y}$. Since both constituent actions are $3$-closed, after composing with an element of $G$, we may assume that $\sigma$ fixes $X$ pointwise and acts on $Y$ as some $g\in G$. For distinct $a,b\in X$, the mixed triples $(a,b,y)$ show that $y^g\in y^{G_{a,b}}$ for every $y\in Y$. Here $G_{a,b}\cong V_4$, and the seven such subgroups, indexed by the lines of the Fano plane, each have two orbits of size four on $Y$. Under a suitable labeling of $Y=\{\infty,0,\ldots,6\}$, the orbits containing $\infty$ are
\[
\begin{gathered}
\{\infty,0,2,6\},\quad \{\infty,1,2,4\},\quad \{\infty,2,3,5\},\quad \{\infty,0,1,3\},\\
\{\infty,0,4,5\},\quad \{\infty,1,5,6\},\quad \{\infty,3,4,6\}.
\end{gathered}
\]
The membership vectors of the eight points in these seven sets are distinct. Thus a permutation preserving every one of these sets is trivial, so $g=1$. The same argument applies to the line action by duality.

Therefore the diagonal action on every pair of maximal coset actions is $3$-closed. Combining this with Proposition~\ref{prop:base-two} for nonmaximal stabilizers, the result follows from Proposition~\ref{prop:two-orbit}.
\end{proof}

\begin{theorem}\label{thm:psl211}
The group $\PSL_2(11)$ is totally $3$-closed.
\end{theorem}

\begin{proof}
Put $G=\PSL_2(11)$. The maximal subgroups of $G$ are two conjugacy classes of subgroups isomorphic to $A_5$, one class of subgroups isomorphic to $C_{11}\rtimes C_5$, and one class of subgroups isomorphic to $D_{12}$ \cite{AtlasFiniteGroups}. A direct intersection calculation shows that every coset action other than the two actions of degree eleven and the action of degree twelve has a base of size at most two. Such actions are $3$-closed by Lemma~\ref{lem:base-size}.

The two degree-eleven actions are the point and block actions of the unique $2$-$(11,5,2)$ biplane, whose full automorphism group is $G$ \cite[Section~3, p.~5]{AlaviDaneshkhahPraeger}. Let $\mathcal R$ be the relation on triples of distinct points defined by requiring the three points to lie in a block. We claim that the blocks are exactly the five-subsets all of whose triples lie in $\mathcal R$.

Indeed, distinct blocks meet in two points. Let $C$ be such a five-subset, choose three points of $C$, and let $B$ be the unique block containing them. If $|B\cap C|=4$, the six unordered pairs in $B\cap C$ give six distinct blocks through the remaining point of $C$, whereas each point lies in only five blocks. If $|B\cap C|=3$, the unique second blocks through the three pairs in $B\cap C$ all contain the other two points of $C$, contradicting the fact that two points lie in exactly two blocks. Hence $C=B$.

Thus every element of the $3$-closure preserves the blocks, and hence belongs to $G$. The point action is $3$-closed, and the block action is $3$-closed by duality.

Now consider the natural action of $G$ on $\PP^1(\FF_{11})$, and let $\sigma$ lie in its $3$-closure. After composing with an element of $G$, we may assume that $\sigma$ fixes $\infty$. For $a\neq b$ in $\FF_{11}$, the $G$-orbit of $(\infty,a,b)$ is determined by the square class of $b-a$. Hence
\[b-a\in(\FF_{11}^\times)^2 \quad\Longleftrightarrow\quad b^\sigma-a^\sigma\in(\FF_{11}^\times)^2. \]
Therefore $\sigma|_{\FF_{11}}$ is an automorphism of the Paley tournament on $\FF_{11}$. The automorphism group of this tournament consists of the maps
\[x\longmapsto ux+v, \qquad u\in(\FF_{11}^{\times})^2,\quad v\in\FF_{11},\]
by \cite[Section~9.7]{JonesPaley}. This is precisely the stabilizer of $\infty$ in $G$, so the projective-line action is $3$-closed.

It remains to consider the diagonal actions on unions of these three coset actions. If $H,K\cong A_5$, then
\[\frac{|H||K|}{|G|}=\frac{60^2}{660}=\frac{60}{11}\notin\ZZ,\]
so $G\neq HK$. Similarly, if $|H|=|K|=55$, then
\[\frac{|H||K|}{|G|}=\frac{55^2}{660}=\frac{605}{132}\notin\ZZ,\]
and again $G\neq HK$. Hence the diagonal action on any union of two
degree-eleven actions, as well as on the union of two degree-twelve actions, is $3$-closed by Proposition~\ref{prop:nonfactor}.

Let $X=G/A_5$, and let $Y=\PP^1(\FF_{11})$. Take $\sigma\in G^{(3),X\sqcup Y}$. Since both constituent actions are $3$-closed, we may normalize so that $\sigma$ fixes $X$ pointwise and its restriction to $Y$ is induced by some $h\in G$.

For each $x\in X$, choose a subgroup $P\leq G_x$ of order $5$. Then $P=G_{y,z}$ for some distinct $y,z\in Y$. Applying the definition of the $3$-closure to $(x,y,z)$ gives $a\in G_x$ such that $y^a=y^h$, and $ z^a=z^h$. Thus $ah^{-1}\in G_{y,z}=P\leq G_x$, and hence $h\in G_x$. Since this holds for every $x\in X$, $h\in\bigcap_{x\in X}G_x=1$. Therefore the mixed action on $X\sqcup Y$ is $3$-closed.

Every pair of coset actions is now $3$-closed by Proposition~\ref{prop:base-two}, Proposition~\ref{prop:nonfactor}, or the preceding argument. The result follows from Proposition~\ref{prop:two-orbit}.
\end{proof}

\begin{theorem}\label{thm:psl213}
The group $\PSL_2(13)$ is totally $3$-closed.
\end{theorem}

\begin{proof}
Put $G=\PSL_2(13)$. By Dickson's classification \cite[Theorem~A.1]{GuralnickZieve}, the maximal proper subgroups of $G$ are
\[C_{13}\rtimes C_6,\qquad D_{12},\qquad D_{14},\qquad A_4.\]

The coset actions corresponding to $D_{12}$ and $D_{14}$ have bases of size two. For $A_4$, there are $91$ conjugates, and at most $18$ of them meet $A_4$ nontrivially. Hence the action on $G/A_4$ also has a base of size two. The same holds for every subgroup contained in one of these groups.

Let $U\cong C_{13}$ and, for $d\mid 6$, put $H_d=U\rtimes C_d$. The action on $G/H_1$ has a base of size two. It remains to consider the actions corresponding to $H_2,H_3$, and $H_6$.

The action on $G/H_6\cong\PP^1(\FF_{13})$ is $3$-closed. Indeed, after fixing $\infty$, the orbits of the triples $(\infty,a,b)$ determine whether $b-a$ is a square. Thus a normalized element of the $3$-closure induces an automorphism of the Paley graph on $\FF_{13}$. Its automorphism group consists of the maps
\[x\longmapsto ux+v, \qquad u\in(\FF_{13}^{\times})^2,\quad v\in\FF_{13}\]
\cite[Theorem~9.1]{JonesPaley}. This is precisely the stabilizer of $\infty$ in $G$.

Now let $E\leq\FF_{13}^{\times}$ have order $4$ or $6$, and write $X_E=(\FF_{13}^2\setminus\{0\})/E$. These are the actions corresponding to $H_2$ and $H_3$. Let $\sigma\in G^{(3),X_E}$. Since $\det(v,w)=0$ if and only if $[v]$ and $[w]$ lie over the same point of $\PP^1(\FF_{13})$, the pair orbits recover the natural projection $X_E\longrightarrow\PP^1(\FF_{13})$. The induced permutation of the projective line belongs to $G$ by the preceding argument, so we may normalize it to be trivial.

The restriction of $\sigma$ to the fiber over a line $L$ is then multiplication by some $c_L\in\FF_{13}^{\times}/E$. For distinct lines $L$ and $M$, preservation of the determinant class gives $c_Lc_M=1$. Using three distinct lines, all $c_L$ are equal to some $c$ with $c^2=1$.

If $|E|=4$, then $\FF_{13}^{\times}/E$ has order $3$, so $c=1$. If $|E|=6$, suppose that $c$ is the nontrivial element of $\FF_{13}^{\times}/E$. Applying the definition of the $3$-closure to the triple $([e_1],[e_2],[e_1+e_2])$ yields an element of $G$ which agrees with $\sigma$ on this triple and therefore fixes the three corresponding projective points. Such an element is trivial, and hence cannot induce the nontrivial scalar class $c$. Thus $c=1$ in this case as well, and both actions are $3$-closed.

Every remaining pair of coset actions either contains an action with a base of size at most two, in which case Proposition~\ref{prop:base-two} applies, or has stabilizers among $H_2,H_3,H_6$. No two of these subgroups factorize $G$, so Proposition~\ref{prop:nonfactor} applies. The result follows from Proposition~\ref{prop:two-orbit}.
\end{proof}

\begin{theorem}\label{thm:psl219}
The group $\PSL_2(19)$ is totally $3$-closed.
\end{theorem}

\begin{proof}
Put $G=\PSL_2(19)$, let $U\cong C_{19}$, and set $H_d=U\rtimes C_d$ for $d\mid 9$. By Lemmas~\ref{lem:psl2-subgroups} and~\ref{lem:base-two-psl2}, every proper subgroup of $G$ is either conjugate to some $H_d$, contained in a subgroup whose coset action has a base of size at most two, or isomorphic to $A_5$. The first two types are covered by Proposition~\ref{prop:borel-fiber} and by Lemma~\ref{lem:base-size} together with Proposition~\ref{prop:base-two}.

There are two conjugacy classes of $A_5$ subgroups, each giving a degree-$57$ coset action. It remains to show that these two actions are $3$-closed and to consider their unions with the actions on $G/H_3$ and $G/H_9$. The action on $G/H_1$ has a base of size two and is already covered by Proposition~\ref{prop:base-two}.

Let $X$ be the coset space corresponding to either conjugacy class of $A_5$ subgroups. Join two points of $X$ when their stabilizers intersect in a subgroup isomorphic to $D_{10}$. The resulting graph $\Gamma$ is the Perkel graph \cite[pp.~163--165]{VandenCruyce1985}. It has intersection array $\{6,5,2;1,1,3\}$ and full automorphism group $G$ \cite[Section~5.8.5, pp.~70--71]{AlfuraidanHall2009}. The edge set of $\Gamma$ is a union of $G$-orbits on pairs, so $G^{(2),X}\leq\Aut(\Gamma)=G$. Both degree-$57$ actions are therefore $2$-closed, and hence $3$-closed.

The diagonal action on any two degree-$57$ actions is $3$-closed by Proposition~\ref{prop:nonfactor}. Indeed, if $A,B\cong A_5$, then
\[\frac{|A||B|}{|G|}=\frac{60^2}{3420}=\frac{20}{19}\notin\ZZ,\]
so $G\neq AB$.

If one of the two coset actions has a base of size at most two, then their diagonal action is $3$-closed by Proposition~\ref{prop:base-two}. This applies in particular to the action on $G/H_1$, where $H_1\cong C_{19}$, since two distinct Sylow $19$-subgroups intersect trivially.

It remains to pair an $A_5$ action $X$ with $Y=G/H_d$, where $d=3$ or $9$. Both actions are $3$-closed. Let $\sigma\in G^{(3),X\sqcup Y}$. After composing with an element of $G$, we may assume that $\sigma$ fixes $Y$ pointwise and that its restriction to $X$ is induced by some $h\in G$.

The subdegrees of the degree-$57$ action are $1,6,20,30$, and the two-point stabilizers on the suborbit of length $20$ have order $3$ \cite[pp.~163--165]{VandenCruyce1985}. Thus every subgroup of order $3$ in $G$ is the stabilizer of two points of $X$.

For $y\in Y$, choose $P\leq G_y$ of order $3$ and points $x_1,x_2\in X$ such that $G_{x_1,x_2}=P$. Applying the definition of the $3$-closure to $(x_1,x_2,y)$ gives $a\in G$ such that
\[x_1^a=x_1^h,\qquad x_2^a=x_2^h,\qquad y^a=y.\]
Hence $a\in G_y$ and
\[ah^{-1}\in G_{x_1,x_2}=P\leq G_y,\]
so $h\in G_y$. Since this holds for every $y\in Y$ and the action on $Y$ is faithful, $h=1$. Therefore the action on $X\sqcup Y$ is $3$-closed.

The pairs of actions on $G/H_d\sqcup G/H_e$ are handled by Proposition~\ref{prop:borel-fiber}. Hence every two-constituent coset action is $3$-closed, and the result follows from Proposition~\ref{prop:two-orbit}.
\end{proof}

Combining the previous results, we obtain part (1) of Theorem~\ref{main}.

\begin{theorem}\label{thm:rank-one}
    Let $G = \PSL_2(q)$ be a finite nonabelian simple group. Then $G$ is totally $3$-closed if and only if $q\geq 7$ is a prime. 
\end{theorem}

\begin{proof}
Theorems~\ref{thm:large-prime}, \ref{thm:psl27}, \ref{thm:psl211}, \ref{thm:psl213}, and~\ref{thm:psl219} show that $\PSL_2(p)$ is totally $3$-closed for every prime $p\geq7$. Proposition~\ref{prop:extension-field} excludes every proper prime power, while Proposition~\ref{prop:q5} excludes $q=5$. Finally, $\PSL_2(2)$ and $\PSL_2(3)$ are not nonabelian simple.
\end{proof}

\section{The $\PSL_3(q)$ classification}\label{sec:psl3}

In this section, we prove that $\PSL_3(q)$ is totally $3$-closed when $q=3$ and when $q$ is a prime with $q\equiv 2\pmod 3$. The converse direction is proved in Section~\ref{sec:higher}.

\subsection{The case $q=3$}
We begin with $q=3$ and put $G=\PSL_3(3)$. Since $|Z(\SL_3(3))|=\gcd(3,3-1)=1$, we have $G=\SL_3(3)$. Moreover, $|G|=3^3(3^2-1)(3^3-1)=5616=2^4\cdot 3^3\cdot 13$, and $G$ is nonabelian simple \cite{AtlasFiniteGroups}. We first record the maximal subgroups of $G$. 

\begin{lemma}\label{lem:psl33-maximal}
Up to conjugacy, the maximal subgroups of $G$ are
\[P,\qquad P^*,\qquad J\cong C_{13}\rtimes C_3,
\qquad M\cong S_4,\]
where $P$ and $P^*$ are respectively a point stabilizer and a line stabilizer in $\operatorname{PG}(2,3)$. Their orders and indices are
\[
\begin{array}{c|cccc}
H & P & P^* & J & M\\
\hline
|H| & 432 & 432 & 39 & 24 \\
{[G:H]} & 13 & 13 & 144 & 234
\end{array}
\]
\end{lemma}

\begin{proof}
The maximal subgroups of $G$ consist of two conjugacy classes of subgroups of type $3^2{:}2S_4$, one class of type $13{:}3$, and one class of type $S_4$ \cite{AtlasOnline}. The two subgroups of order $432$ are respectively the point and line stabilizers in $\operatorname{PG}(2,3)$. The displayed orders and indices now follow from $|G|=5616$.
\end{proof}

We first treat the $J$ and $M$ cases by counting conjugates. 

\begin{lemma}\label{lem:psl33-low-base}
If $H$ is contained in a conjugate of $J$ or $M$, then the action of $G$ on $G/H$ has a base of size at most two.
\end{lemma}

\begin{proof}
It suffices to find $g\in G$ such that
\[J\cap J^g=1 \qquad\text{and}\qquad M\cap M^g=1.\]

There are $144$ conjugates of $J$. Let $R\cong C_{13}$ be its normal Sylow $13$-subgroup. Every nonidentity element of $R$ lies in a unique conjugate of $J$, since $N_G(R)=J$. The other $26$ elements of $J$ have order $3$ and lie in a conjugacy class of size $624$ \cite[table \texttt{L3(3)} and the fusion from \texttt{13:3}]{CTblLib}. Hence each lies in $\frac{144\cdot 26}{624}=6$ conjugates of $J$. Therefore at most $26(6-1)=130<143$ other conjugates meet $J$ nontrivially. Thus $J\cap J^g=1$ for some $g\in G$.

There are $234$ conjugates of $M$. The group $M\cong S_4$ contains nine involutions and eight elements of order $3$. These lie in conjugacy classes of sizes $117$ and $624$, respectively \cite[table \texttt{L3(3)} and the fusion from \texttt{s4}]{CTblLib}. Thus an involution lies in $\frac{234\cdot 9}{117}=18$ conjugates of $M$, while an element of order $3$ lies in $\frac{234\cdot 8}{624}=3$ conjugates. Every nontrivial subgroup of $S_4$ contains an element of order $2$ or $3$, so at most $9(18-1)+8(3-1)=169<233$ other conjugates meet $M$ nontrivially. Hence $M\cap M^g=1$ for some $g\in G$.

Finally, if $H\leq J$ or $H\leq M$, the same choice of $g$ gives $H\cap H^g=1$. Therefore the action on $G/H$ has a base of size at most two.
\end{proof}

We next treat the point and line actions. 

\begin{lemma}\label{lem:psl33-projective}
The natural actions of $G$ on the points and lines of $\operatorname{PG}(2,3)$ are $3$-closed.
\end{lemma}

\begin{proof}
Consider first the action on points. Two ordered triples of pairwise distinct points lie in the same $G$-orbit if and only if they are either both collinear or both noncollinear. Since $G$ is $2$-transitive, the $G$-orbits on triples with at most two distinct entries are determined by which coordinates are equal.

Thus every element of the $3$-closure preserves collinearity, and hence acts on $\operatorname{PG}(2,3)$ as a collineation. By the fundamental theorem of projective geometry, every collineation is induced by a semilinear transformation \cite[Theorem~1.9.1]{BallFiniteGeometry}. Since $\Aut(\FF_3)=1$ and $\PGL_3(3)=\PSL_3(3)=G$, the point action is $3$-closed. The same argument applied to the dual projective plane proves that the line action is $3$-closed.
\end{proof}

It remains to treat the subgroups of $P$ and $P^*$ whose coset actions do not have a base of size two. Fix a point stabilizer $P$ and put
\[\mathcal E(P)=\{H<P:H\cap H^g\neq1\text{ for every }g\in G\}.\]
Thus $H<P$ lies outside $\mathcal E(P)$ if and only if the action on $G/H$ has a base of size at most two.

\begin{proposition}\label{prop:psl33-E3}
The set $\mathcal E(P)$ contains exactly twenty-four subgroups, forming six $G$-conjugacy classes represented by subgroups of orders
\[48,\qquad54,\qquad72,\qquad108,\qquad144,\qquad216.\]
Each corresponding coset action is $3$-closed. The same holds for the six dual classes contained in a line stabilizer.
\end{proposition}

\begin{proof}
This proposition is established by an exhaustive computation in \GAP{} \cite{GAP4}, the only such computation in the paper. The complete program is reproduced, together with its output, in Appendix~\ref{app:gap}. The program constructs $G$ in its degree-$13$ point action and enumerates all $646$ subgroups of $P$ by cyclic extension.

For each of the $645$ proper subgroups $H<P$, the program tests one representative of every double coset in $H\backslash G/H$. Since $|H\cap H^g|$ is constant as $g$ ranges over a double coset $HgH$, this test is exhaustive. Of these subgroups, $621$ have a conjugate intersecting $H$ trivially. Conjugacy tests in $G$ show that the remaining twenty-four form six $G$-conjugacy classes, one for each of the orders in the statement.

Let $\Omega=G/H$, let $x=H\in\Omega$ be the point corresponding to the trivial coset, and let $A_x$ be the group of permutations of $\Omega$ that fix $x$ and preserve every $G_x$-orbit on $\Omega\times\Omega$. Every element of $G^{(3),\Omega}$ fixing $x$ preserves each $G$-orbit of triples with first entry $x$, and therefore lies in $A_x$. The program determines $A_x$ by exact backtracking and then removes every element that fails to preserve some $G$-orbit on $\Omega^3$:
\[
\begin{array}{c|rrrrrr}
|H|&48&54&72&108&144&216\\ \hline
\#\{K\in\mathcal E(P):|K|=|H|\}&9&4&3&4&3&1\\
|A_x|&48&54&72&108&288&1944\\
|(G^{(3),\Omega})_x|&48&54&72&108&144&216
\end{array}
\]
For the first four classes, we already have $A_x=G_x=H$. For the final two classes, the group $A_x$ is larger, but every element of $A_x\setminus G_x$ fails to preserve some $G$-orbit on $\Omega^3$. Hence $(G^{(3),\Omega})_x=H$ in all six cases.

Since $G\leq G^{(3),\Omega}$, the group $G^{(3),\Omega}$ is transitive. Therefore
\[|G^{(3),\Omega}|=|\Omega|\,|(G^{(3),\Omega})_x|=|G:H|\,|H|=|G|.\]
Thus $G^{(3),\Omega}=G$.

Finally, the automorphism $g\mapsto(g^{-1})^{\mathsf T}$ interchanges point and line stabilizers, giving the six dual classes.
\end{proof}

We now turn to diagonal actions on two coset spaces. 

\begin{lemma}\label{lem:psl33-sync}
Let $\Omega_1$ and $\Omega_2$ be faithful transitive $3$-closed $G$-sets. Suppose that there are surjective $G$-maps $\rho_i:\Omega_i\longrightarrow Y_i$, where each $Y_i$ is either the point set or the line set of $\operatorname{PG}(2,3)$. Then the diagonal action of $G$ on $\Omega_1\sqcup\Omega_2$ is $3$-closed.
\end{lemma}

\begin{proof}
Let $\sigma\in G^{(3),\Omega_1\sqcup\Omega_2}$. By \cite[Lemma~2.2]{FreedmanGiudiciPraeger2024}, its restrictions to $\Omega_1$ and $\Omega_2$ lie in the corresponding $3$-closures. Thus they are induced by elements $a,b\in G$. After composing with the diagonal action of $a^{-1}$, we may assume that $\sigma$ fixes $\Omega_1$ pointwise and acts on $\Omega_2$ as some $t\in G$.

By \cite[Lemma~2.1]{FreedmanGiudiciPraeger2024}, $\sigma$ preserves every $G$-orbit on $\Omega_1\times\Omega_2$. If $Y_1$ and $Y_2$ are of the same type, apply this to the $G$-invariant relation $\rho_1(x)=\rho_2(y)$. Since $\rho_1$ and $\rho_2$ are surjective, it follows that $t$ fixes every point of $Y_2$.

If $Y_1$ and $Y_2$ are of opposite types, consider the $G$-invariant relation defined by requiring the point $\rho_1(x)$ to lie on the line $\rho_2(y)$, or vice versa. Since $\sigma$ fixes $\Omega_1$ pointwise, it follows that $t$ preserves every point-line incidence involving an element of $Y_1$. Surjectivity of $\rho_1$ and $\rho_2$ then implies that $t$ fixes every element of $Y_2$.

The point and line actions of $G$ are faithful, so $t=1$. Hence $\sigma$ is induced by an element of $G$, and the diagonal action is $3$-closed.
\end{proof}

\begin{theorem}\label{thm:psl33-main}
    The group $G = \PSL_3(3)$ is totally $3$-closed. 
\end{theorem}

\begin{proof}
Let $H,K<G$ be proper subgroups. By Lemma~\ref{lem:psl33-maximal}, each is contained in a conjugate of $P$, $P^*$, $J$, or $M$.

If $H$ is contained in a conjugate of $J$ or $M$, then the action on $G/H$ has a base of size at most two by Lemma~\ref{lem:psl33-low-base}. The same holds if $H$ is contained in a conjugate of $P$ or $P^*$ but does not belong to one of the exceptional classes in Proposition~\ref{prop:psl33-E3}. Otherwise, the action on $G/H$ is the point action, the line action, or one of the exceptional parabolic actions. These actions are $3$-closed by Lemmas~\ref{lem:base-size} and~\ref{lem:psl33-projective} and Proposition~\ref{prop:psl33-E3}. The same conclusions hold for $K$.

If either action has a base of size at most two, then the diagonal action on $G/H\sqcup G/K$ is $3$-closed by Proposition~\ref{prop:base-two}. Otherwise, both actions admit surjective $G$-maps onto the point set or the line set of $\operatorname{PG}(2,3)$, so Lemma~\ref{lem:psl33-sync} applies.

Thus the diagonal action on $G/H\sqcup G/K$ is $3$-closed for every pair of proper subgroups $H,K<G$. The result follows from Proposition~\ref{prop:two-orbit}.
\end{proof}

\subsection{The primes $p\equiv 2\pmod 3$} In this subsection, we prove that if $p$ is a prime with $p\equiv 2\pmod 3$, then $\PSL_3(p)$ is totally $3$-closed. The remaining prime power cases are treated together with higher ranks in Section~\ref{sec:higher}.

For $p=2$, the result follows from $\PSL_3(2)\cong\PSL_2(7)$ and Theorem~\ref{thm:psl27}. Throughout the remainder of this subsection, $p$ denotes a prime with $p\geq 5$ and $p\equiv 2\pmod3$. Since $|Z(\SL_3(p))|=\gcd(3,p-1)=1$, we have $G=\PSL_3(p)=\SL_3(p)$.

We begin with a rigidity statement for the scalar-fiber actions.

\begin{proposition}\label{prop:psl3-scalar-fiber}
Let $p\geq 5$ be a prime with $p\equiv2\pmod3$, let $G=\PSL_3(p)=\SL_3(p)$, and let $C\leq\FF_p^\times$. Put
\[\Omega_C=(\FF_p^3\setminus\{0\})/C.\]
Then the natural action of $G$ on $\Omega_C$ is faithful and $3$-closed.
\end{proposition}

\begin{proof}
Let $g$ lie in the kernel of the action. Then $g$ fixes every point of $\PP^2(\FF_p)$, so $g$ is scalar. Since $\det(g)=1$ and $\gcd(3,p-1)=1$, it follows that $g=1$. Thus the action is faithful.

Let $\pi:\Omega_C\longrightarrow\PP^2(\FF_p)$ be the natural map, and let $\sigma\in G^{(3),\Omega_C}$. Two elements of $\Omega_C$ have the same image under $\pi$ if and only if their representatives are linearly dependent. Since $\sigma$ preserves the $G$-orbits on ordered pairs, it preserves the fibers of $\pi$ and induces a permutation of $\PP^2(\FF_p)$.

The induced permutation preserves collinearity, since three projective points are collinear if and only if their representatives are linearly dependent. By the fundamental theorem of projective geometry, it is induced by a semilinear transformation \cite[Theorem~1.9.1]{BallFiniteGeometry}. Since $\FF_p$ has no nontrivial field automorphisms and $\PGL_3(p)=\PSL_3(p)=G$, the induced permutation belongs to $G$. After composing $\sigma$ with the inverse of this element, we may assume that $\sigma$ fixes every fiber of $\pi$ setwise.

Choose a nonzero representative $v_L$ for each $L\in\PP^2(\FF_p)$, and put $A=\FF_p^\times/C$. The fiber over $L$ is identified with $A$ by $a\longmapsto [av_L]$. Thus $\sigma$ induces a bijection $f_L:A\to A$ on each fiber.

Let $L_1,L_2,L_3$ be noncollinear. The $G$-orbit of
\[([a_1v_{L_1}],[a_2v_{L_2}],[a_3v_{L_3}])\]
is determined by $a_1a_2a_3\in A$. Hence
\[f_{L_1}(a_1)f_{L_2}(a_2)f_{L_3}(a_3)=a_1a_2a_3.\]
Varying one coordinate shows that $f_L(a)=u_La$ for some $u_L\in A$. Moreover, $u_{L_1}u_{L_2}u_{L_3}=1$ whenever $L_1,L_2,L_3$ are noncollinear.

Given $L,L'\in\PP^2(\FF_p)$, choose distinct points $M,N$ on a line containing neither $L$ nor $L'$. Comparing the triples $(L,M,N)$ and $(L',M,N)$ gives $u_L=u_{L'}$. Thus all $u_L$ are equal to some $u\in A$, and $u^3=1$. Since $|A|\mid p-1$ and $3\nmid p-1$, we have $u=1$. Therefore $\sigma=1$ after composition, and the original permutation belongs to $G$.
\end{proof}

Now we treat the nonparabolic maximal subgroups.

\begin{proposition}
\label{prop:psl3-nonparabolic-host}
Let $H$ be contained in a nonparabolic maximal subgroup of $G$. Then the action of $G$ on $G/H$ has a base of size at most two.
\end{proposition}

\begin{proof}
Let $H\leq M<G$, where $M$ is nonparabolic and maximal. It is enough to find $g\in G$ such that $M\cap M^g=1$. Indeed, this then gives $H\cap H^g=1$.

By the subgroup classification for $\PSL_3(p)$ \cite[Theorem~2.4]{King2005}, the subgroup $M$ is one of the following:
\begin{enumerate}
\item[\textup{(i)}] the stabilizer $N_2$ of a projective triangle;
\item[\textup{(ii)}] the normalizer $N_3$ of a Singer cycle;
\item[\textup{(iii)}] a subgroup isomorphic to $\SO_3(p)$;
\item[\textup{(iv)}] a subgroup isomorphic to $\PSL_2(7)$.
\end{enumerate}
The other cases in the classification do not occur because $p$ is prime and $3\nmid p-1$.

Suppose first that $M=N_2$. Let $\Delta=\{\langle e_1\rangle,\langle e_2\rangle,\langle e_3\rangle\}$, and put
\[v_1=e_1+e_2,\qquad v_2=e_2+2e_3,\qquad v_3=e_1+e_2+e_3.\]
The matrix
\[
g_2=
\begin{pmatrix}
1&0&1\\
1&1&1\\
0&2&1
\end{pmatrix}
\]
has determinant one, and $N_2^{g_2}$ stabilizes $\Delta'=\{\langle v_1\rangle,\langle v_2\rangle,\langle v_3\rangle\}$. Let $x\in N_2\cap N_2^{g_2}$. Write $x=DP_\tau$, where $D$ is diagonal and $\tau\in S_3$. Since $\langle v_3\rangle$ is the unique point of $\Delta'$ whose coordinates are all nonzero, $x$ fixes $\langle v_3\rangle$. Hence $D$ is scalar. The only coordinate permutations preserving the supports of $\langle v_1\rangle$ and $\langle v_2\rangle$ are $1$ and $(1\,3)$, while
\[(1\,3)\langle v_1\rangle=\langle e_2+e_3\rangle\neq \langle e_2+2e_3\rangle.\]
Thus $\tau=1$, so $x=dI$. Since $d^3=1$ and $\gcd(3,p-1)=1$, we have $x=1$. Therefore $N_2\cap N_2^{g_2}=1$.

Now suppose that $M=N_3$. Identify $\FF_p^3$ with $K=\FF_{p^3}$ and put
\[T=\{x\mapsto ax:N_{K/\FF_p}(a)=1\}.\]
If $\sigma:x\mapsto x^p$, then $N_3=T\rtimes\langle\sigma\rangle$. Set $n=p^2+p+1$, so $|T|=n$ and $|N_3|=3n$. Since $3\nmid n$, every element of $N_3\setminus T$ has order three.

For $1\neq x\in T$, one has
\[|C_G(x)|=n\qquad\text{and}\qquad |x^G\cap N_3|=3.\]
Thus
\[
\bigl|\{g\in G:x\in N_3^g\}\bigr|=|x^G\cap N_3|\,|C_G(x)|=3n.
\]
For $y\in N_3\setminus T$,
\[|C_G(y)|=p^2-1 \qquad\text{and}\qquad |y^G\cap N_3|=2n,\]
so that
\[
\bigl|\{g\in G:y\in N_3^g\}\bigr|=2n(p^2-1).
\]
Summing over the $n-1$ nontrivial elements of $T$ and the $2n$ elements of $N_3\setminus T$ gives
\[ \bigl|\{g\in G:N_3\cap N_3^g\neq1\}\bigr| \leq 3n(n-1)+4n^2(p^2-1). \]
Since $|G|=p^3(p-1)n(p^2-1)$ and
\[\frac{|G|-3n(n-1)-4n^2(p^2-1)}{n}=p^6-p^5-5p^4-3p^3-3p^2+p+4>0\]
for $p\geq5$, there exists $g_3\in G$ such that $N_3\cap N_3^{g_3}=1$.

Suppose next that $M=\SO(V,B)$ for a nondegenerate symmetric bilinear form $B$. Identify $V$ with $K=\FF_{p^3}$. After an isometry, we may write $B(x,y)=\operatorname{Tr}_{K/\FF_p}(cxy)$ for some $c\in K^\times$. Choose $1\neq\alpha\in K^\times$ with $N_{K/\FF_p}(\alpha)=1$, and let $A$ be multiplication by $\alpha$. Then $A$ is self-adjoint with respect to $B$ and has irreducible minimal polynomial of degree three.

Define $B_A(x,y)=B(Ax,y)$. The forms $B$ and $B_A$ have the same discriminant, so there exists $g\in G$ such that $\SO(B_A)=\SO(B)^g$. If $x\in\SO(B)\cap\SO(B_A)$, then $xA=Ax$. Hence $x$ is multiplication by some $\beta\in K^\times$. Preservation of $B$ gives $\beta^2=1$, so $x=\pm I$. Since $x\in G$ and $\det(-I)=-1$, we obtain $x=I$. Thus $\SO(B)\cap\SO(B)^g=1$.

Finally, suppose that $M\cong\PSL_2(7)$. This case can occur only for $p\geq11$, so $p\nmid |M|$ and every nonidentity element of $M$ is semisimple. The numbers of elements of orders $2,3,4$, and $7$ in $M$
are
\[21,\qquad56,\qquad42,\qquad48.\]
Every noncentral semisimple element of $G$ has centralizer of order at most $p(p-1)^2(p+1)$. Therefore
\[\begin{aligned}
\bigl|\{g\in G:M\cap M^g\neq1\}\bigr|&\leq (21^2+56^2+42^2+48^2)p(p-1)^2(p+1)\\
&=7645p(p-1)^2(p+1).
\end{aligned}\]
Since $7645<p^2(p^2+p+1)$ for $p\geq11$, this number is smaller than $|G|=p^3(p^3-1)(p^2-1)$. Hence $M\cap M^g=1$ for some $g\in G$.

Thus in every case there exists $g\in G$ such that $M\cap M^g=1$.
It follows that $H\cap H^g=1$, so the action on $G/H$ has a base of size at most two.
\end{proof}

We next prove a lemma used in the treatment of the parabolic subgroups.

\begin{lemma}\label{lem:levi-separator}
Let $p\geq5$, let $L=\GL_2(p)$, and let $M\leq L$ with $\SL_2(p)\nleq M$. If $D$ is a split maximal torus of $L$, then there exists $g\in L$ such that
\[D\cap M^g=D\cap\operatorname{core}_L(M).\]
\end{lemma}

\begin{proof}
Let $Z=Z(L)$, let $\pi:L\longrightarrow\PGL_2(p)$, and put $A=\pi(M)$. We first claim that $A$ does not contain $\PSL_2(p)$. Otherwise,
\[\pi([M,M])=[A,A]\geq\PSL_2(p).\]
Since $[M,M]\leq\SL_2(p)$, the subgroup $M\cap\SL_2(p)$ maps onto $\PSL_2(p)$. It is therefore either $\SL_2(p)$ or a subgroup of index two. The latter is impossible because $\SL_2(p)$ is perfect, while the former contradicts the hypothesis.

By the subgroup classification of $\PGL_2(p)$ \cite[Theorem~D]{Faber}, the group $A$ is cyclic, dihedral, elementary abelian of order four, isomorphic to $A_4$, $S_4$, or $A_5$, or contained in an affine subgroup. We show that $A$ has a base of size at most two on $\PP^1(\FF_p)$.

An affine subgroup has the base $(0,1)$. A subgroup of a split torus normalizer has the base $(1,c)$ for any $c\neq0,\pm1$. For a nonsplit torus normalizer, every nonidentity torus element fixes no point of $\PP^1(\FF_p)$, while each element outside the torus fixes at most two points. Since there are $p+1$ such elements and $2(p+1)<p(p+1)$, some ordered pair of distinct points has trivial stabilizer. A subgroup of order four is contained in the centralizer of one of its involutions, which is a torus normalizer.

Finally, let $E\cong A_4,S_4$, or $A_5$. Every nonidentity element of $\PGL_2(p)$ fixes at most two points, and hence at most two ordered pairs of distinct points. Thus $E$ has a base of size two whenever $2(|E|-1)<p(p+1)$. This holds for $A_4$ when $p\geq5$, for $S_4$ when $p\geq7$, and for $A_5$ when $p\geq11$. When $p=5$, the case $A_5=\PSL_2(5)$ is excluded. For $S_4\leq\PGL_2(5)$, the transformations
\[a(z)=2z,\qquad b(z)=\frac{z+2}{z+3}\]
generate a subgroup isomorphic to $S_4$. Its stabilizer of $\infty$ is $\langle a\rangle$, and no nonidentity element of $\langle a\rangle$ fixes $1$. Hence $(\infty,1)$ is a base. Since $\PGL_2(5)$ has a unique conjugacy class of $S_4$ subgroups \cite[Theorem~D\textup{(7)}]{Faber}, the claim follows.

We now determine the core of $M$. The group $\operatorname{core}_{\PGL_2(p)}(A)$ is trivial. Indeed, its intersection with the simple normal subgroup $\PSL_2(p)$ is trivial, and it therefore centralizes $\PSL_2(p)$. Since $C_{\PGL_2(p)}(\PSL_2(p))=1$, the claim follows. Hence $\operatorname{core}_L(M)=M\cap Z$.

Let $(x,y)$ be a base for $A$ on $\PP^1(\FF_p)$. Choose $g\in L$ such that $(x,y)^g$ is the pair of eigenlines of $D$. Then $\pi(D)\cap\pi(M^g)=1$, and therefore $D\cap M^g\leq Z$. Since $M\cap Z$ is central and $Z\leq D$, we have $M\cap Z\leq D\cap M^g$. Consequently,
\[D\cap M^g=M\cap Z = \operatorname{core}_L(M)=D \cap \operatorname{core}_L(M).\]
\end{proof}

Now we are ready to prove the parabolic cases.

\begin{lemma}\label{lem:psl3-parabolic-descendants}
Let $P$ be a point or line stabilizer of $G$, and let $H\leq P$. Then the action of $G$ on $G/H$ is $3$-closed.
\end{lemma}

\begin{proof}
By contragredient duality, it is enough to consider a point stabilizer. Put
\[\Omega=G/H, \qquad \widehat G=G^{(3),\Omega}, \qquad P^0=Q\rtimes\SL_2(p), \qquad Q\cong\FF_p^2, \]
and set $J=HP^0$. Since $P^0\trianglelefteq P$, the subgroup $J$ is well defined.

The partition of $\Omega$ into the $J$-orbits is determined by a
$G$-invariant relation on ordered pairs. Hence $\widehat G$ preserves
this partition and induces a permutation group on $G/J$.
Moreover, this induced group is contained in the $3$-closure of the
action of $G$ on $G/J$.

Let
\[C=J/P^0\leq P/P^0\cong\FF_p^\times.\]
The action on $G/J$ is isomorphic to the action on $(\FF_p^3\setminus\{0\})/C$. It is therefore $3$-closed by Proposition~\ref{prop:psl3-scalar-fiber}. Thus, for $c\in\widehat G$, we may compose $c$ with an element of $G$ and obtain $k\in\widehat G$ fixing every $J$-orbit setwise.

For each projective point $x$, let $X_x$ be the corresponding $J$-orbit in $\Omega$, let $P_x$ be its stabilizer in $G$, and let $P_x^0$ be the corresponding conjugate of $P^0$. Then $\alpha^k\in\alpha^{P_x^0}$ for $\alpha\in X_x$. Choose $H_x\leq P_x$ corresponding to the action of $P_x$ on $X_x$, and put $N_x=\operatorname{core}_{P_x}(H_x)$.

Write $P_x=Q_x\rtimes L_x$ where $L_x\cong\GL_2(p)$. The action of $L_x$ on $Q_x$, written as row vectors, is $\rho(A)u=(\det A)^{-1}uA^{-1}$. The subgroup $N_x\cap Q_x$ is $L_x$-invariant, so it is either $1$ or $Q_x$.

If $N_x\cap Q_x=1$, then
\[[N_x,Q_x]\leq N_x\cap Q_x=1.\]
Thus $N_x\leq C_{P_x}(Q_x)=Q_x$, and hence $N_x=1$.

Suppose that $Q_x\leq N_x$. Then $N_x/Q_x$ is normal in $L_x$. Such a normal subgroup either contains $\SL_2(p)$ or is scalar. Consequently, either $P_x^0\leq N_x$, or $N_x=1$, or $N_x=Q_x\rtimes Z_x$ for some scalar subgroup $Z_x\leq L_x$. We say that $N_x$ is \emph{small} in the latter two cases.

If $N_x$ and $N_y$ are small and $x\neq y$, then $N_x\cap N_y=1$. Indeed, taking $x=\langle e_1\rangle$ and $y=\langle e_2\rangle$, an element of the intersection has both forms
\[
\begin{pmatrix}
c^{-2}&u&v\\
0&c&0\\
0&0&c
\end{pmatrix},
\qquad
\begin{pmatrix}
d&0&0\\
u'&d^{-2}&v'\\
0&0&d
\end{pmatrix}.
\]
Equality gives $c=d=c^{-2}$ and forces the off-diagonal entries to vanish. Thus $c^3=1$, so $c=1$. Since the pairs $(P_x,H_x)$ are $G$-conjugate, the same alternative holds for every $x$.

Suppose first that $N_x$ is small for every $x$. Let $F=(x,y,z)$ be an ordered projective basis, and put
\[D_F=P_x\cap P_y\cap P_z, \qquad K_x=D_F\cap N_x.\]
We claim that there exists $\eta_x\in X_x$ such that $(D_F)_{\eta_x}=K_x$.

Identify $P_x=Q\rtimes L$, let $M$ be the image of $H_x$ in $L$, and put $W=H_x\cap Q$. If $\SL_2(p)\leq M$, then the smallness of $N_x$ gives $W=0$. Since $-I$ acts as $-1$ on $Q$, every complement to $Q$ in $Q\rtimes\SL_2(p)$ is $Q$-conjugate to the standard complement. We may therefore assume that $H_x\leq L$.

The two characters through which $D_F$ acts on $Q$ have exponent matrix of determinant $3$. Since $\gcd(3,p-1)=1$, the resulting map $D_F\longrightarrow(\FF_p^\times)^2$ is bijective. A vector whose two coordinates are nonzero therefore gives a point $\eta_x$ with trivial $D_F$-stabilizer.

Now suppose that $\SL_2(p)\nleq M$. By Lemma~\ref{lem:levi-separator}, after conjugating in $L$ we may assume that
\[D_F\cap M=D_F\cap\operatorname{core}_L(M)=:R.\]
If $W=Q$, the point corresponding to $H_x$ has stabilizer $D_F\cap H_x=R=K_x$. If $W<Q$, then $N_x=1$ and $R$ is scalar. For
$1\neq d=cI\in R$, the map $\rho(d)-1=(c^{-3}-1)I$ is invertible. The condition that $d$ belong to a $Q$-conjugate of $H_x$ excludes at most one coset of $W$ in $Q$. There are at most $|R|-1\leq p-2$ such excluded cosets, whereas $|Q/W|\geq p$. A conjugate outside their union gives the required point $\eta_x$.

Choose such points $\eta_x,\eta_y,\eta_z$ for the ordered basis $F$. Since $k$ preserves their $G$-orbit and fixes their three projective points, there exists $t_F\in D_F$ such that
\[(\eta_x,\eta_y,\eta_z)^k=(\eta_x,\eta_y,\eta_z)^{t_F}.\]
Set $h_F=kt_F^{-1}$. Then $h_F$ fixes the three chosen points.

Since $G$ and $\widehat G$ have the same orbits on $\Omega^3$, repeated application of \cite[Theorem~2.3(ii)]{OBrienPonomarenkoVasilievVdovin} shows that $G_{\eta_y,\eta_z}$ and $\widehat G_{\eta_y,\eta_z}$ have the same orbits on $\Omega$. Let $u\in X_x$. There exists $a\in G_{\eta_y,\eta_z}$ such that $u^{h_F}=u^a$. Both points lie in $X_x$, so $a$ fixes $x$. Since $a$ also fixes $\eta_y$ and $\eta_z$, it fixes $y$ and $z$. Hence
\[a\in D_F\cap G_{\eta_y}\cap G_{\eta_z}= K_y\cap K_z \leq N_y\cap N_z = 1.\]
Thus $h_F$ fixes $X_x$ pointwise. By symmetry, it fixes $X_y$ and $X_z$ pointwise. Therefore $k$ and $t_F$ induce the same permutation on these three sets.

If two ordered projective bases $F$ and $F'$ agree in two entries, then $t_Ft_{F'}^{-1}$ fixes the corresponding two sets pointwise. Hence $t_Ft_{F'}^{-1}\in N_x\cap N_y=1$. The graph of ordered projective bases, in which two bases are adjacent when they agree in two entries, is connected. Thus all $t_F$ are equal to one element $t\in G$.

Every projective point occurs in an ordered projective basis, and $t$ fixes every point in each such basis. Hence $t$ fixes every point of $\PP^2(\FF_p)$. It is therefore scalar, and so $t=1$ because $Z(G)=1$. The sets $X_x$ cover $\Omega$, so $k=1$.

Finally, suppose that $P_x^0\leq N_x$ for every $x$. Since $P_x^0$ acts trivially on $X_x$, the relation $\alpha^k\in\alpha^{P_x^0}$ implies $\alpha^k=\alpha$ for every $\alpha\in X_x$. Again the sets $X_x$ cover $\Omega$, so $k=1$.

Thus every element of $\widehat G$ belongs to $G$, and the action on $G/H$ is $3$-closed.
\end{proof}

We now combine the preceding results. 

\begin{lemma}\label{lem:psl3-all-constituents}
For every proper subgroup $H<G$, the action of $G$ on $G/H$ is $3$-closed. If $H$ is contained in a nonparabolic maximal subgroup, then this action has a base of size at most two.
\end{lemma}

\begin{proof}
Choose a maximal subgroup $M$ of $G$ containing $H$. If $M$ is parabolic, then the action on $G/H$ is $3$-closed by
Lemma~\ref{lem:psl3-parabolic-descendants}.

If $M$ is nonparabolic, then the action on $G/H$ has a base of size at most two by Proposition~\ref{prop:psl3-nonparabolic-host}. It is therefore $3$-closed by Lemma~\ref{lem:base-size}.
\end{proof}

\begin{lemma}\label{lem:psl3-parabolic-sync}
Let $H\leq P_X$ and $K\leq P_Y$, where $P_X$ and $P_Y$ are point or line parabolic subgroups of $G$. Then the diagonal action of $G$ on $G/H\sqcup G/K$ is $3$-closed.
\end{lemma}

\begin{proof}
By Lemma~\ref{lem:psl3-parabolic-descendants}, the actions on $G/H$ and $G/K$ are $3$-closed. There are surjective $G$-maps
\[\rho_X:G/H\longrightarrow G/P_X, \qquad \rho_Y:G/K\longrightarrow G/P_Y,\]
where $G/P_X$ and $G/P_Y$ are respectively point or line sets of $\operatorname{PG}(2,p)$.

Let $\pi\in G^{(3),G/H\sqcup G/K}$. Constant triples show that $\pi$ preserves each of the two coset spaces. By \cite[Lemma~2.2]{FreedmanGiudiciPraeger2024}, its restrictions are induced by elements $a,b\in G$. After composing $\pi$ with $a^{-1}$, we may assume that $\pi$ fixes $G/H$ pointwise and acts on $G/K$ as $\delta=a^{-1}b\in G$.

Let $x\in G/P_X$ and $y\in G/P_Y$, and choose $\alpha\in\rho_X^{-1}(x)$ and $\beta\in\rho_Y^{-1}(y)$. By \cite[Lemma~2.1]{FreedmanGiudiciPraeger2024}, there exists $t\in G$ such that $\alpha^t=\alpha$ and $\beta^t=\beta^\delta$. Hence $x^t=x$ and $y^t=y^\delta$.

Suppose first that $G/P_X$ and $G/P_Y$ are both point sets or both line sets. Since equality is preserved by $G$, we have
\[x=y \quad\Longleftrightarrow\quad x=y^\delta\]
for every $x$. Taking $x=y$ gives $y^\delta=y$.

Suppose now that one is the point set and the other is the line set. Since incidence is preserved by $G$, for every point $x$ and line $y$,
\[x\in y \quad\Longleftrightarrow\quad x\in y^\delta.\]
A line is determined by its incident points, and a point is determined by its incident lines. Thus again $y^\delta=y$.

Therefore $\delta$ fixes every point or every line of $\operatorname{PG}(2,p)$. Since these actions of $G$ are faithful, $\delta=1$. Hence $\pi$ is induced by the diagonal action of $a\in G$, and the action on $G/H\sqcup G/K$ is $3$-closed.
\end{proof}

\begin{theorem}\label{thm:psl3-prime}
Let $p$ be a prime with $p\equiv2\pmod3$. Then $\PSL_3(p)$ is totally $3$-closed.
\end{theorem}

\begin{proof}
For $p=2$, the result follows from $\PSL_3(2)\cong\PSL_2(7)$ and Theorem~\ref{thm:psl27}. Assume that $p\geq 5$.

Let $H,K<G$ be proper subgroups. By Lemma~\ref{lem:psl3-all-constituents}, the actions on $G/H$ and $G/K$ are $3$-closed. Choose maximal subgroups $M_H$ and $M_K$ containing $H$ and $K$, respectively.

If either $M_H$ or $M_K$ is nonparabolic, then the corresponding coset action has a base of size at most two by Lemma~\ref{lem:psl3-all-constituents}. Hence the diagonal action on $G/H\sqcup G/K$ is $3$-closed by Proposition~\ref{prop:base-two}.

If both $M_H$ and $M_K$ are parabolic, then the diagonal action is $3$-closed by Lemma~\ref{lem:psl3-parabolic-sync}.

Thus the diagonal action on $G/H\sqcup G/K$ is $3$-closed for every pair of proper subgroups $H,K<G$. The result follows from Proposition~\ref{prop:two-orbit}.
\end{proof}

\section{Negative results and the proof of the main theorem}\label{sec:higher}
In this section, we establish the semilinear and central-vector obstructions and complete the proof of Theorem~\ref{main}. In particular, we prove that $\PSL_4(q)$ is not totally $3$-closed for any prime power $q$, and that $\PSL_n(q)$ is not totally $3$-closed whenever $n\geq5$ and $q>2$. This leaves only the family $\PSL_n(2)$ with $n\geq5$.

\begin{proposition}
\label{prop:projective-local}
Let $n\geq3$, let $q$ be a prime power, and let $V=\FF_q^n$. Put
\[G=\PSL(V),\qquad K=\PGL(V),\qquad \Omega=\PP(V).\]
Then $\PGammaL(V)\leq G^{(3),\Omega}$.
\end{proposition}

\begin{proof}
Define $\delta:K\longrightarrow \FF_q^\times/(\FF_q^\times)^n$ by $\delta([A])=\det(A)(\FF_q^\times)^n$. This is well defined and surjective, with kernel $G$.

Let $T=(L_1,L_2,L_3)\in\Omega^3$, and let $K_T$ be its pointwise
stabilizer in $K$. We claim that $\delta(K_T)=\FF_q^\times/(\FF_q^\times)^n$. If the distinct lines in $T$ span a subspace of dimension at most two, scale a complementary basis vector by an arbitrary element of $\FF_q^\times$. If $L_1,L_2,L_3$ are independent, choose nonzero $v_i\in L_i$, extend them to a basis of $V$, and scale $v_1$. In either case, the resulting element fixes $T$ and has arbitrary determinant class.

Suppose that $U=T^k$ for some $k\in K$. Choose $s\in K_U$ such that $\delta(s)=\delta(k)^{-1}$. Then $ks\in G$ and $T^{ks}=U$. Thus $G$ and $K$ have the same orbits on $\Omega^3$.

Two ordered triples lie in the same $K$-orbit if and only if the same coordinates are equal and, when the three points are distinct, they are either both collinear or both noncollinear. Every semilinear projectivity preserves these properties. Hence every element of $\PGammaL(V)$ preserves every $G$-orbit on $\Omega^3$, and therefore $\PGammaL(V)\leq G^{(3),\Omega}$.
\end{proof}

\begin{theorem}\label{thm:semilinear-main}
Let $n\geq3$ and let $q=r^f$ be a prime power. If $f>1$ or $\gcd(n,q-1)>1$, then the natural action of $\PSL_n(q)$ on $\PP^{n-1}(\FF_q)$ is not $3$-closed. In particular, $\PSL_n(q)$ is not totally $3$-closed.
\end{theorem}

\begin{proof}
Put $G=\PSL_n(q)$ and $ \Omega=\PP^{n-1}(\FF_q)$. By Proposition~\ref{prop:projective-local}, $\PGammaL_n(q)\leq G^{(3),\Omega}$.

Suppose first that $d=\gcd(n,q-1)>1$. Then $[\PGL_n(q):G]=d>1$. Since
\[G<\PGL_n(q)\leq\PGammaL_n(q)\leq G^{(3),\Omega}, \]
the action is not $3$-closed.

If $f>1$, the Frobenius permutation
\[[x_1:\cdots:x_n]\longmapsto [x_1^r:\cdots:x_n^r]\]
belongs to $\PGammaL_n(q)$ but not to $\PGL_n(q)$. Thus in either case $G<G^{(3),\Omega}$, so the action is not $3$-closed.
\end{proof}

\begin{theorem}
\label{thm:central-vector-main}
Let $n\geq4$, let $q>2$ be a prime power, let $V=\FF_q^n$, and let $Z=Z(\SL(V))$. Put
\[G=\SL(V)/Z,\qquad \widetilde G=\GL(V)/Z,\qquad \Omega=(V\setminus\{0\})/Z.\]
Then the actions of $G$ and $\widetilde G$ on $\Omega$ are faithful, and $\widetilde G\leq G^{(k),\Omega}$ for every $1\leq k<n$. In particular, $G\cong\PSL_n(q)$ is not totally $3$-closed.
\end{theorem}

\begin{proof}
Let $A\in\GL(V)$ fix every $Z$-orbit in $V\setminus\{0\}$. For each
$0\neq v\in V$, write $Av=z_vv$ for some $z_v\in Z$. If $u,v$ are linearly independent, then
\[z_uu+z_vv = A(u+v) = z_{u+v}(u+v), \]
so $z_u=z_v=z_{u+v}$. The same scalar occurs on every nonzero multiple of a vector. Hence $A=zI$ for some $z\in Z$. Thus the actions of $G$ and $\widetilde G$ on $\Omega$ are faithful.

Since $Z\leq\SL(V)$ and $\GL(V)/\SL(V)\cong\FF_q^\times$, we have $[\widetilde G:G]=q-1>1$.

Let $1\leq k<n$, let $A\in\GL(V)$, and consider $(Zv_1,\ldots,Zv_k)\in\Omega^k$. Put $W=\langle v_1,\ldots,v_k\rangle$. Since $\dim W<n$, choose a basis of $V$ extending a basis of $W$. There exists $D\in\GL(V)$ which fixes $W$ pointwise and satisfies $\det(D)=\det(A)^{-1}$. Indeed, it is enough to scale one basis vector outside $W$. Then $AD\in\SL(V)$, and $AD$ agrees with $A$ on the given tuple. Hence every element of $\widetilde G$ can be matched on every ordered $k$-tuple by an element of $G$. Therefore $\widetilde G\leq G^{(k),\Omega}$.

Taking $k=3$, we obtain $G<\widetilde G\leq G^{(3),\Omega}$, so the action is not $3$-closed.
\end{proof}

\begin{proposition}\label{prop:psl42-main}
The group $\PSL_4(2)$ is not totally $3$-closed.
\end{proposition}

\begin{proof}
Identify $\PSL_4(2)\cong A_8$ and let $A_8$ act on the set $\Omega$ of two-subsets of $\{1,\ldots,8\}$ \cite[p.~3]{AtlasFiniteGroups}. This action is faithful.

Let $T=(X_1,X_2,X_3)\in\Omega^3$. The union $X_1\cup X_2\cup X_3$ contains at most six points. Let $s\in S_8$ and put $U=T^s$. Choose two points outside the union of the three entries of $U$, and let $\tau$ be the transposition of these two points. Then $\tau$ fixes $U$ pointwise. If $s$ is odd, then $s\tau$ is even and
\[T^{s\tau}=U^\tau=U=T^s.\]
Hence every $S_8$-orbit on $\Omega^3$ is an $A_8$-orbit. Therefore $S_8\leq A_8^{(3),\Omega}$. Since $A_8<S_8$, the action is not $3$-closed. Thus $\PSL_4(2)$ is not totally $3$-closed.
\end{proof}

We can now complete the classification.

\begin{proof}[Proof of Theorem~\textup{\ref{main}}]
Part~(1) is Theorem~\ref{thm:rank-one}.

For $n=3$, the positive cases follow from
Theorems~\ref{thm:psl33-main} and~\ref{thm:psl3-prime}. If $q=r^f$ with $f>1$, or if $q$ is prime and $q\equiv1\pmod3$, the group is not totally $3$-closed by Theorem~\ref{thm:semilinear-main}. This proves part~(2).

Suppose that $n=4$. If $q>2$, the result follows from Theorem~\ref{thm:central-vector-main}. If $q=2$, it follows from Proposition~\ref{prop:psl42-main}. This proves part~(3).

Finally, if $n\geq5$ and $q>2$, the result follows from Theorem~\ref{thm:central-vector-main}. This proves part~(4). Hence only the cases $\PSL_n(2)$ with $n\geq5$ remain.
\end{proof}

\appendix

\section{The GAP verification}\label{app:gap}

The following \GAP{} program carries out the computation used in the proof of Proposition~\ref{prop:psl33-E3}. It constructs $G=\PSL_3(3)$ in its action on the thirteen points of $\operatorname{PG}(2,3)$, enumerates the $646$ subgroups of the point stabilizer $P$, determines the twenty-four exceptional subgroups and their six $G$-conjugacy classes, and computes the groups $A_x$ and $(G^{(3),\Omega})_x$ for one representative of each class. The program was executed under \GAP{}~4.12.1 \cite{GAP4} and runs in under ten seconds.

\begin{lstlisting}[caption={The verification program \texttt{psl33\_exceptional\_gap.txt}.},label={lst:gap}]
# psl33_exceptional.g
#
# Verification of the exceptional-subgroup classification for
# G = PSL(3,3) = SL(3,3): enumerates all subgroups of a point
# stabilizer P, determines the set E(P) of subgroups H < P with
# H meeting every G-conjugate of itself, classifies E(P) up to
# G-conjugacy, and verifies that each exceptional coset action
# G/H is 3-closed.
#
# Run as:  gap -q -A -b psl33_exceptional.g

# G acting on the 13 points of PG(2,3), P = stabilizer of point 1.
pts := Set(Filtered(GF(3)^3, v -> not IsZero(v)), NormedRowVector);;
G := Image(ActionHomomorphism(SL(3,3), pts, OnLines));;
P := Stabilizer(G, 1);;
Print("|G| = ", Size(G), "  |P| = ", Size(P), "\n");

# All subgroups of P, by cyclic extension.
subs := Concatenation(List(ConjugacyClassesSubgroups(
          LatticeByCyclicExtension(P)), AsList));;
Print("subgroups of P: ", Length(subs), "\n");

# |H meet H^g| is constant on the double coset HgH, so one
# representative per double coset is an exhaustive test.
IsExceptional := function(H)
  local g;
  if Size(H) = 1 or Size(H) = Size(P) then return false; fi;
  for g in List(DoubleCosets(G, H, H), Representative) do
    if IsTrivial(Intersection(H, H^g)) then return false; fi;
  od;
  return true;
end;;
exc := Filtered(subs, IsExceptional);;
Print("exceptional subgroups: ", Length(exc), "\n");
Print("order distribution: ", Collected(List(exc, Size)), "\n");

# G-conjugacy classes of the exceptional subgroups.
classes := [];;
for H in exc do
  cl := First(classes, c -> Size(c[1]) = Size(H)
                            and IsConjugate(G, c[1], H));
  if cl = fail then Add(classes, [H, 1]);
  else cl[2] := cl[2] + 1; fi;
od;
Print("classes [order, members in P]: ",
      List(classes, c -> [Size(c[1]), c[2]]), "\n");

# For x = H in Omega = G/H, the group A_x of permutations of Omega
# fixing x and preserving every G_x-orbit on Omega^2 contains the
# stabilizer of x in G^(3). Backtracking enumerates A_x; the
# transported pair coloring then filters by the G-orbits on Omega^3:
# the G-orbit of (i,j,k) is determined by D[T_i(j)][T_i(k)], where
# T_i carries i back to x and D is the G_x-orbit coloring of pairs.
ClosureStabilizer := function(H)
  local act, GG, N, stab, orb, D, ncol, rp, Ti, i, o,
        A2, A3, img, used, search, pi, ok, j, k;
  act := FactorCosetAction(G, H);
  GG := Image(act);
  N := Index(G, H);
  stab := Stabilizer(GG, 1);
  orb := OrbitsDomain(stab, Cartesian([1..N], [1..N]), OnPairs);
  D := List([1..N], i -> []);
  for i in [1..Length(orb)] do
    for o in orb[i] do D[o[1]][o[2]] := i; od;
  od;
  ncol := Length(orb);
  rp := List([1..N], i -> RepresentativeAction(GG, 1, i));
  Ti := List([1..N], i -> List([1..N], j -> j^(rp[i]^-1)));
  A2 := [];
  img := [1];
  used := BlistList([1..N], [1]);
  search := function(v)
    local w, a, good;
    if v > N then Add(A2, PermList(ShallowCopy(img))); return; fi;
    for w in [1..N] do
      if not used[w] then
        good := true;
        for a in [1..v-1] do
          if D[a][v] <> D[img[a]][w]
             or D[v][a] <> D[w][img[a]] then
            good := false; break;
          fi;
        od;
        if good then
          img[v] := w; used[w] := true;
          search(v+1);
          used[w] := false; Unbind(img[v]);
        fi;
      fi;
    od;
  end;
  search(2);
  if Length(A2) = Size(H) then
    A3 := A2;
  else
    A3 := [];
    for pi in A2 do
      ok := true;
      for i in [1..N] do
        for j in [1..N] do
          for k in [1..N] do
            if D[Ti[i][j]][Ti[i][k]]
               <> D[Ti[i^pi][j^pi]][Ti[i^pi][k^pi]] then
              ok := false; break;
            fi;
          od;
          if not ok then break; fi;
        od;
        if not ok then break; fi;
      od;
      if ok then Add(A3, pi); fi;
    od;
  fi;
  return [Size(H), N, ncol, Length(A2), Length(A3),
          Set(A3) = Set(Elements(stab))];
end;;

Print("[ |H|, |Omega|, pair colors, |A_x|, |3-closure stab|, ",
      "= image of H ]\n");
for cl in classes do
  Print(ClosureStabilizer(cl[1]), "\n");
od;
Print("total time (ms): ", Runtime(), "\n");
QUIT;
\end{lstlisting}

The program produces the following output.

\begin{lstlisting}[language={},numbers=none,caption={Output of Listing~\ref{lst:gap}.},label={lst:gapout}]
|G| = 5616  |P| = 432
subgroups of P: 646
exceptional subgroups: 24
order distribution: [ [ 48, 9 ], [ 54, 4 ], [ 72, 3 ], [ 108, 4 ], 
  [ 144, 3 ], [ 216, 1 ] ]
classes [order, members in P]: [ [ 48, 9 ], [ 54, 4 ], [ 72, 3 ], [ 108, 4 ], 
  [ 144, 3 ], [ 216, 1 ] ]
[ |H|, |Omega|, pair colors, |A_x|, |3-closure stab|, = image of H ]
[ 48, 117, 345, 48, 48, true ]
[ 54, 104, 264, 54, 54, true ]
[ 72, 78, 148, 72, 72, true ]
[ 108, 52, 72, 108, 108, true ]
[ 144, 39, 22, 288, 144, true ]
[ 216, 26, 16, 1944, 216, true ]
total time (ms): 9482
\end{lstlisting}

The six rows of the final output table match the table in the proof of Proposition~\ref{prop:psl33-E3}. In each exceptional class, the stabilizer of $x$ in the $3$-closure has order $|H|$ and coincides with the image of $H$. An independent implementation of the same computation in \textsf{Python}, produced during the Albilich run that generated the original argument, returns identical results.

\section*{Report of AI use}
The main argument was generated with the assistance of Albilich \cite{gong2026albilichsteerableproofstateorchestration}, a generative artificial-intelligence assistant for mathematical research developed by the authors. All arguments were subsequently fully understood, completely rewritten, and independently verified by the authors.

\bibliographystyle{amsplain}
\bibliography{references}

\end{document}